\documentclass[12pt]{article}
\usepackage{amssymb}
\usepackage{amsmath,amsthm}
\usepackage[latin1]{inputenc}
\usepackage{graphicx}
\usepackage{hyperref}
\usepackage{enumerate}
\usepackage{tikz}
\usepackage{mathrsfs}
\usepackage{verbatim}

\hypersetup{colorlinks=true, linkcolor=blue, citecolor=blue, urlcolor=blue}

\newtheorem{remark}{Remark}

\newtheorem{lemma}[remark]{Lemma}
\newtheorem{theorem}[remark]{Theorem}
\newtheorem{proposition}[remark]{Proposition}
\newtheorem{corollary}[remark]{Corollary}
\newtheorem{claim}[remark]{Claim}

\newcommand{\adim}{\operatorname{\mathrm{adim}}}

\title{The $k$-metric dimension of the lexicographic product of graphs}

\author{A. Estrada-Moreno$^{(1)}$, I. G. Yero$^{(2)}$, J. A. Rodr\'{i}guez-Vel\'{a}zquez$^{(1)}$\\
$^{(1)}${\small Departament d'Enginyeria Inform\`atica i Matem\`atiques,}\\
{\small Universitat Rovira i Virgili,}  {\small Av. Pa\"{\i}sos
Catalans 26, 43007 Tarragona, Spain.} \\{\small
 alejandro.estrada\@@urv.cat, juanalberto.rodriguez\@@urv.cat}\\
    $^{(2)}${\small Departamento de Matem\'aticas, Escuela Polit\'ecnica Superior de Algeciras}\\
{\small Universidad de C\'adiz,} {\small
Av. Ram\'on Puyol s/n, 11202 Algeciras, Spain.} \\ {\small
ismael.gonzalez\@@uca.es}
}

\begin{document}
\maketitle

\begin{abstract}
Given a simple and connected  graph $G=(V,E)$, and a positive integer $k$, a set $S\subseteq V$ is said to be a $k$-metric generator for $G$, if  for any pair of different vertices $u,v\in V$, there exist at least $k$ vertices $w_1,w_2,\ldots,w_k\in S$ such that $d_G(u,w_i)\ne d_G(v,w_i)$, for every $i\in \{1,\ldots,k\}$, where $d_G(x,y)$ denotes the distance between $x$ and $y$. The minimum cardinality of a $k$-metric generator is the $k$-metric dimension of $G$. A set $S\subseteq V$ is a $k$-adjacency generator for $G$ if any two different vertices $x,y\in V(G)$ satisfy $|((N_G(x)\triangledown N_G(y))\cup\{x,y\})\cap S|\ge k$, where $N_G(x)\triangledown N_G(y)$ is the symmetric difference of the neighborhoods of $x$ and $y$.
The minimum cardinality of any $k$-adjacency generator is the $k$-adjacency dimension of $G$. In this article we obtain tight bounds and closed formulae for the $k$-metric dimension of the lexicographic product of graphs in terms of the $k$-adjacency dimension of the factor graphs.
\end{abstract}



\section{Introduction}

Locating sets in graphs were introduced in \cite{Slater1975}, in connection to some usefulness of such a sets into long range aids to navigation. Moreover, the same structure was defined independently in \cite{Harary1976} under the name of resolving sets. Further on, in \cite{Khuller1996}, the (locating or resolving) sets were renamed as metric generators which is a more intuitive definition, according to the role they play inside the graph. This last name arise from the concept of metric generators of metric spaces. That is, given a simple and connected graph $G=(V,E)$, a vertex $v\in V$ is said to \emph{distinguish} two vertices $x$ and $y$ if $d_G(v,x)\ne d_G(v,y)$, where $d_G(x,y)$ is the length of a shortest path between $x$ and $y$. A set $S\subset V$ is said to be a \emph{metric generator} for $G$ if any pair of vertices of $G$ is distinguished by some element of $S$. In this sense, if we consider the metric $d_G:V\times V\rightarrow \mathbb{N}$, then $(V,d_G)$ is clearly a metric space. A metric generator with the smallest possible cardinality among all the metric generators for $G$ is called a \emph{metric basis} of  $G$, and its cardinality the \emph{metric dimension} of $G$, denoted by $\dim(G)$.

A model of application of metric generators to navigation of robots in networks was also described in  \cite{Khuller1996}. In this model, the robot moves from node to node of a graph and can locate themselves throughout a uniquely distinctive labeled ``landmark'' node set $S$ of the graph. It is assumed that a robot moving over the graph is using the distances to the landmarks  to ``knows'' its position in each moment, \emph{i.e.}, if a robot knows the distances to the vertices of $S$, then its position on the graph is uniquely determined. With this purpose, the set of landmarks $S$ is a metric generator for the graph modeling the network topology. A very important goal is then to minimize the number of landmarks needed, and to determine where they should be located, so that the distances to the landmarks uniquely determine the robot's position on the graph. Solutions to these questions are produced by the metric dimension and some metric basis of the graph, respectively. Although these concepts solve the above described problem, there is a weakness which has not been taken into account, \emph{i.e.}, the possible uniqueness of a landmark which could distinguish some pairs of vertices. For more realistic settings, in \cite{Estrada-Moreno2013}, the concept of $k$-metric generator was introduced and some preliminary properties of it were studied. Moreover, in the article \cite{Yero2013c}, some complexity issues on that new parameter were considered.

Let $G=(V,E)$ be a simple and connected graph. A set $S\subseteq V$ is said to be a \emph{$k$-metric generator} for $G$ if any pair of vertices of $G$ is distinguished by at least $k$ elements of $S$, {\em i.e.}, for any pair of different vertices $u,v\in V$, there exist at least $k$ vertices $w_1,w_2,\ldots,w_k\in S$ such that
\begin{equation}\label{conditionDistinguish}
d_G(u,w_i)\ne d_G(v,w_i),\; \mbox{\rm for every}\; i\in \{1,\ldots,k\}.
\end{equation}
A $k$-metric generator of minimum cardinality in $G$ is a \emph{$k$-metric basis} and its cardinality the $k$-\emph{metric dimension} of $G$, which is denoted by $\dim_{k}(G)$. Note that every $k$-metric generator $S$ satisfies that $|S|\geq k$ and, if $k>1$, then $S$ is also a $(k-1)$-metric generator. Moreover, $1$-metric generators are the standard metric generators (resolving sets or locating sets as defined in \cite{Harary1976} or \cite{Slater1975}, respectively). In practice, the problem of checking if a set $S$ is a $1$-metric generator reduces to check condition (\ref{conditionDistinguish}) only for those vertices $u,v\in V- S$, as every vertex in $S$ is distinguished at least by itself. Also, if $k=2$, then condition (\ref{conditionDistinguish}) must be checked only for those pairs having at most one vertex in $S$, since two vertices of $S$ are distinguished at least by themselves. Nevertheless, if $k\ge 3$, then condition (\ref{conditionDistinguish}) must be checked for every pair of different vertices of the graph.

It is clear that not for every value $k$ there exists a $k$-metric generator for a graph $G$. This fact allows to give the following definition. A connected graph $G$ is said to be a \emph{$k'$-metric dimensional graph} if $k'$ is the largest integer such that there exists a $k'$-metric basis \cite{Estrada-Moreno2013}. Notice that if $G$ is a $k'$-metric dimensional graph, then for each positive integer $k\le k'$, there exists at least one $k$-metric basis of $G$, \textit{i.e.,} $\dim_{k}(G)$ makes sense for $k\in \{1,\ldots,k'\}$.

It was shown in  \cite{Yero2013c}  that the decision problem, regarding whether the $k$-metric dimension of a graph is less than an specific integer, is NP-complete (the case $k=1$ was previously studied in \cite{Khuller1996}). It is therefore motivating to find the $k$-metric dimension for special classes of graphs or obtaining good bounds on this
invariant. Specifically, for the case of product graphs, it would be desirable to reduce the problem of computing the $k$-metric dimension of a product graph into computing the $k$-metric dimension of the factor graphs. Early studies about the metric dimension of product graphs were initiated in \cite{Caceres2007,Peters-Fransen2006}, where several tight bounds and closed formulae for the metric dimension of Cartesian product graphs were presented. After that, the metric dimension of corona graphs, rooted product graphs, lexicographic product graphs, direct product graphs and strong product graphs was studied in \cite{Yero2011}, \cite{Yero2013b}, \cite{JanOmo2012,Saputro2013}, \cite{Kuziak2015} and \cite{Rodriguez-Velazquez2013a}, respectively. Also, for the $k$-metric dimension in product graphs, some studies in the case of corona product of graphs were presented in \cite{Estrada-Moreno2013corona}. In this paper we continue with the study of the $k$-metric dimension of product graphs, specifically we consider the lexicographic product. For more information on product graph definitions we suggest the book \cite{Hammack2011}.

The \emph{lexicographic product} of a graph $G$ of order $n$ and a family of graphs $\mathcal{H}=\{H_1,H_2,\ldots,H_n\}$, which is denoted by $G\circ\mathcal{H}$, is the graph with vertex set $\bigcup_{v_i\in V(G)}\left(\{v_i\}\times V(H_i)\right)$, where $(a,v)$ is adjacent to $(b,w)$ whenever $ab\in E(G)$, or $a=b$ and $vw\in E(H_i)$ for every $H_i\in\mathcal{H}$. Note that $G\circ\mathcal{H}$ is connected if and only if $G$ is connected. Thus, throughout this paper, we consider that $G$ is connected graph of order $n$ with $V(G)=\{u_1,\ldots,u_n\}$. Further, every $H_i\in\mathcal{H}$ is a graph of order $n_i$ with $V(H_i)=\{v_1^i,\ldots,v_{n_i}^i\}$. Note that this approach of lexicographic product is a natural generalization of the standard lexicographic product of graphs, and therefore its properties too. For more information on the lexicographic product of two graphs we suggest \cite{Hammack2011}. If for every $H_i\in\mathcal{H}$ holds that $H_i\cong H$, then we will use the notation $G\circ H$ (as in the standard case) instead of $G\circ\mathcal{H}$ and we will refer $H_i$ as $i^{\text{th}}$-copy of $H$.

The article is structured in the following way. In Section \ref{sect-prelim} we give some preliminary known results which are necessary for the rest of the following sections. Section \ref{sectionDimensionalLexi} is devoted to compute the value $k$ for which any lexicographic product graph is $k$-metric dimensional or equivalently we give the suitable interval of integer numbers in which the $k$-metric dimension of lexicographic product graphs makes sense. In Section \ref{sect-k-dim} we give tight bounds and closed formulae for the $k$-metric dimension of lexicographic product graphs and we finish the article with a Conclusion section.

\section{Preliminary concepts}\label{sect-prelim}

In this section we include some definitions and known results that are necessary to prove our main results. If two vertices $u,v$ are adjacent in $G=(V,E)$, then we write $u\sim v$ or  $uv\in E(G)$. Given  $x\in V(G)$, we define $N_{G}(x)$ as the \emph{open neighborhood} of $x$ in $G$, \textit{i.e.},  $N_{G}(x)=\{y\in V(G):x\sim y\}$. The \emph{closed neighborhood}, denoted by $N_{G}[x]$, equals $N_{G}(x)\cup \{x\}$. If there is no  ambiguity, we simply write  $N(x)$ or $N[x]$. We also define $\delta(v)=|N(v)|$ as the \emph{degree} of vertex $v$, as well as, $\delta(G)=\min_{v\in V(G)}\{\delta(v)\}$ and $\Delta(G)=\max_{v\in V(G)}\{\delta(v)\}$. For a non-empty set $S \subseteq V(G)$, and a vertex $v \in V(G)$, $N_S(v)$ denotes the set of neighbors that $v$ has in $S$, {\it i.e.}, $N_S(v) = S\cap N(v)$. As usual, we denote by $A\triangledown B=(A\cup B)- (A\cap B)$ the \emph{symmetric difference} of two sets  $A$ and $B$. We use $\overline{G}$ for the \emph{complement} of $G$.

Two vertices $x,y$ are called \emph{false twins} if $N(x)=N(y)$, and $x,y$ are called \emph{true twins} if $N[x]=N[y]$. In particular, if $G$ contains more than
one isolated vertex, then they are false twin vertices. Two different vertices $x,y$ are \emph{twins} if they are either false twin vertices or true twin vertices. We also say that a vertex $x$ is a twin, if there exists other vertex $y$ such that $x,y$ are twins. In concordance with that, we define the \textit{twin  equivalence relation} ${\cal R}$ on $V(G)$ as follows:
$$x {\cal R} y \longleftrightarrow N_G(x)-\{y\}=N_G(y)-\{x\}.$$
We have three possibilities for each twin equivalence class $U$:
\begin{enumerate}[(a)]
\item $U$ is a singleton twin equivalence class, or
\item $U$ is a false twin equivalence class, \textit{i.e.}, $N_G(x)=N_G(y)$, for any $x,y\in U$ (and case (a) does not apply), or
\item $U$ is a true twin equivalence class, \textit{i.e.}, $N_G[x]=N_G[y]$, for any $x,y\in U$ (and case (a) does not apply).
\end{enumerate}
If all twin equivalence classes of a graph $G$ are singletons, then we say that $G$ is a \textit{twins free graph}. If $G$ does not have any true (false) twin equivalence class, then we say that $G$ is a \textit{true $($false$)$ twins free graph}.

Now we define a graph metric which will be  useful throughout the article. Given a connected graph $G$ it is defined the metric $d_{G,2}: V(G)\times V(G)\longmapsto \mathbb{N}$ as $$d_{G,2}(x,y)=\min\{d_G(x,y),2\}.$$
The following claim on vertex distances in the lexicographic product of a graph and a family of graphs, where we use the metric above, is obtained analogously to the standard case of lexicographic product of two graphs \cite{Hammack2011}.

\begin{claim}\label{claimLexi}{\rm \cite{Hammack2011}}
Let $G$ be a connected graph of order $n$ and let $\mathcal{H}$ be a family of $n$ graphs. Then the following statements hold,
\begin{enumerate}[{\rm (i)}]
\item  $d_{G\circ \mathcal{H}}((u_i,v_l^i),(u_j,v_m^j)) = d_{G}(u_i,u_j)$ for $i\ne j$ and $1\le i,j\le n$.
\item  $d_{G\circ \mathcal{H}}((u_i,v_l^i),(u_i,v_m^i)) = d_{H_i,2}(v_l^i,v_m^i)$ for $1\le i\le n$.
\end{enumerate}
\end{claim}

To continue with some necessary results, we give the following concepts from \cite{Estrada-Moreno2014a}. A set $S\subseteq V(G)$ is a $k$-\textit{adjacency generator} for $G$, if for every two vertices $x,y\in V(G)$ there exist at least $k$ vertices $w_1,w_2,\ldots,w_k\in S$, such that
$$d_{G,2}(x,w_i)\ne d_{G,2}(y,w_i),\; \mbox{\rm for every}\; i\in \{1,\ldots,k\},$$
\textit{i.e}, $S$ is a $k$-\textit{adjacency generator} for $G$ if for every two vertices $x,y\in V(G)$ it holds $$|((N_G(x)\triangledown N_G(y))\cup\{x,y\})\cap S|\ge k.$$ A minimum $k$-adjacency generator is called a $k$-\textit{adjacency basis} of $G$ and its cardinality, the $k$-\textit{adjacency dimension} of $G$, is denoted by $\adim_k(G)$. A graph $G$ is said to be a \emph{$k'$-adjacency dimensional graph} if $k'$ is the largest integer such that there exists a $k'$-adjacency basis.

The $k$-adjacency dimension of some specific families of $k'$-adjacency dimensional graphs was studied in \cite{Estrada-Moreno2014a}. As an example, next we present the cases of paths and cycles\footnote{Notice that paths are $3$-adjacency dimensional and cycles are $4$-adjacency dimensional.}, since we will further on use them in this article. On the other hand, since $N_G(x)\triangledown N_G(y)=N_{\overline{G}}(x)\triangledown N_{\overline{G}}(y)$, for any non-trivial graph $G$, it is straightforward to deduce that $$\adim_k(G)=\adim_k(\overline{G})$$ for $\displaystyle k\in\left\{1,2,\ldots,\min_{x,y\in V(G)}\left\{\left|(N_G(x)\triangledown N_G(y))\cup\{x,y\}\right|\right\}\right\}$. It is clear that every $k$-adjacency generator for a graph $G$ is also a $k$-metric generator for $G$, and thus, $\dim_k(G)\le\adim_k(G)$.

\begin{proposition}\label{value-adj-Paths-Cycles}{\rm \cite{Estrada-Moreno2014a}}
For any integer $n\ge 4$,
\begin{enumerate}[{\rm (i)}]
\item $\adim_2(P_n)=\adim_2(\overline{P}_n)=\lceil\frac{n+1}{2}\rceil$,
\item $\adim_3(P_n)=\adim_3(\overline{P}_n)=n-\lfloor\frac{n-4}{5}\rfloor$.
\end{enumerate}
Also, for any $n\ge 5$,
\begin{enumerate}[{\rm (i')}]
\item $\adim_2(C_n)=\adim_2(\overline{C}_n)=\left\lceil\frac{n}{2}\right\rceil$,
\item $\adim_3(C_n)=\adim_3(\overline{C}_n)= n-\left\lfloor\frac{n}{5}\right\rfloor$,
\item $\adim_4(C_n)=\adim_4(\overline{C}_n)=n$.
\end{enumerate}
\end{proposition}

In addition to all the notations and terminology already mentioned, we use the notation $K_n$, $C_n$, $N_n$ and $P_n$ for complete graphs, cycle graphs, empty graphs and path graphs of order $n$, respectively, and $K_{r,s}$ for complete bipartite graphs of order $r+s$. In this work, the remaining definitions will be given the first time that the concept appears in the text.

\section{Computing the value $k$ for which $G\circ\mathcal{H}$ is $k$-metric dimensional}\label{sectionDimensionalLexi}

In order to compute the $k$-metric dimension of a graph, it is necessary to know the interval of possible values for $k$, for which this can be computed. Since in this article we deal with the $k$-metric dimension of the lexicographic product of graphs, the first issue which we need to solve is precisely finding the value $k$ for which $G\circ\mathcal{H}$ is $k$-metric dimensional. We recall that this could be computed in polynomial time according to an algorithm given in \cite{Yero2013c}. Nevertheless we particularize that general studies for the purposes of this article.

The following concepts were already defined in \cite{Estrada-Moreno2013,Yero2013c}. Given two vertices $x,y\in V(G)$, the set of \textit{distinctive vertices} of $x,y$ is $${\cal D}_G(x,y)=\left\{z\in V(G): d_{G}(x,z)\ne d_{G}(y,z)\right\}$$
and, the set of \emph{non-trivial distinctive vertices} of  $x,y$ is $${\cal D}^*_G(x,y)={\cal D}_G(x,y)-\{x,y\}.$$ Finally, according to the local definitions above, it is defined the following global parameter,
$${\cal D}(G)=\min_{x,y\in V(G)}|{\cal D}_G(x,y)|.$$
The polynomial algorithm given in \cite{Yero2013c} for computing the value $k$ for which $G\circ\mathcal{H}$ is $k$-metric dimensional is based on the following result, which use the parameter above.

\begin{theorem} \label{theokmetric}  {\rm \cite{Estrada-Moreno2013}}
A connected graph  $G$ is $k$-metric dimensional  if and only if $k={\cal D}(G)$.
\end{theorem}

According to the result above, given a connected graph $G$ of order $n$ and a family $\mathcal{H}$ of $n$ non-trivial graphs, our goal in this section is to find the value ${\cal D}(G\circ \mathcal{H})$ and express this in terms of some known parameters of $G$ and $\mathcal{H}$.

Now, analogously to the definitions above for the case of $k$-metric generators, in \cite{Estrada-Moreno2014a} were presented the following concepts, regarding the $k$-adjacency generators of graphs. Given two vertices $x,y\in V(G)$, the set of \textit{adjacency distinctive vertices} of $x,y$ is $${\cal C}_G(x,y)=(N_G(x)\triangledown N_G(y))\cup\{x,y\},$$
and, the set of \emph{non-trivial adjacency distinctive vertices} of  $x,y$ is $$\mathcal{C}_G^*(x,y)=\mathcal{C}_G(x,y)-\{x,y\}.$$ From the above, we define the following global parameter, $$\mathcal{C}(G)=\min_{x,y\in V(G)}\{|\mathcal{C}_G(x,y)|\}.$$
In clear analogy with Theorem \ref{theokmetric} the following formula, from \cite{Estrada-Moreno2014a}, gives the value $k$ for which a graph is $k$-adjacency dimensional.

\begin{theorem}\label{theokadjacency} {\rm \cite{Estrada-Moreno2014a}}
A graph $G$ is $k$-adjacency dimensional if and only if $k=\mathcal{C}(G)$.
\end{theorem}

Twin vertices plays a highly significant role into studying the $k$-metric dimension of graphs, as we will observe through our exposition. In this sense, we need to use some more formal terminology regarding themselves. Given a vertex $x\in V(G)$, we define the \emph{true twin equivalence class} to which $x$ belongs by $TT(x)$, and we define the \emph{false twin equivalence class} to which $x$ belongs by $FT(x)$. Also we denote by $S(G)$, $FT(G)$ and $TT(G)$ the union of singletons, the false, and the true twin equivalence classes of a graph $G$, respectively. Now, for any graph $G$ of order $n$, a family of $n$ graphs $\mathcal{H}=\{H_1,\ldots,H_n\}$ and $u_i\in V(G)$, we define in $G\circ\mathcal{H}$ the following local parameter:
$$
\mathcal{T}(u_i,\mathcal{H})=\left\lbrace
\begin{array}{ll}
|V(H_i)|,&\text{if } u_i\in S(G),\\
\\
\displaystyle\min_{u_j,u_l\in FT(u_i)}\{\delta(H_j)+\delta(H_l)+2\}, & \text{if } u_i\in FT(G),\\
\\
\displaystyle\min_{u_j,u_l\in TT(u_i)}\{|V(H_j)|-\Delta(H_j)+|V(H_l)|-\Delta(H_l)\}, & \text{if } u_i\in TT(G).
\end{array}
\right.
$$

Moreover, we define a global parameter from the local parameter defined above,
$$\mathcal{T}(G\circ\mathcal{H})=\min_{u_i\in V(G)}\{\mathcal{T}(u_i,\mathcal{H})\}.$$ We also define $$\mathcal{C}(\mathcal{H})=\min_{H_i\in\mathcal{H}}\{\mathcal{C}(H_i)\}.$$

With all the tools presented till this point, we are now prepare to give our first result regarding the value $k$ for which the lexicographic product graphs are $k$-metric dimensional.

\begin{theorem}\label{theoLexiDimensional}
Let $G$ be a connected graph of order $n\ge 2$ and let $\mathcal{H}$ be a family of $n$ non-trivial graphs. The graph $G\circ\mathcal{H}$ is $k$-metric dimensional if and only if $k=\min\{\mathcal{T}(G\circ\mathcal{H}),\mathcal{C}(\mathcal{H})\}$.
\end{theorem}

\begin{proof}
By Theorem \ref{theokmetric}, it is only necessary to prove that $\mathcal{D}(G\circ\mathcal{H})=\min\{\mathcal{T}(G\circ\mathcal{H}),\mathcal{C}(\mathcal{H})\}$. Hence, let $(u_i,v_x^i),(u_j,v_y^j)\in V(G\circ\mathcal{H})$ be two different vertices. We analyze two cases.\\
\\Case 1. $i=j$. By Claim \ref{claimLexi} (i) and (ii), it follows that $\mathcal{D}_{G\circ\mathcal{H}}((u_i,v_x^i),(u_i,v_y^i))=\{u_i\}\times\mathcal{C}_{H_i}(v_x^i,v_y^i)$. Thus, $$R_1=\min_{(u_i,v_x^i),(u_i,v_y^i)\in V(G\circ\mathcal{H})}\{|\mathcal{D}_{G\circ\mathcal{H}}((u_i,v_x^i),(u_i,v_y^i))|\}=\min_{H_i\in\mathcal{H}}\{\mathcal{C}(H_i)\}=\mathcal{C}(\mathcal{H}).$$
Case 2. $i\ne j$. If $u_i,u_j$ are not twins, then $\mathcal{D}_G^*(u_i,u_j)\ne\emptyset$. So, for every $u_l\in\mathcal{D}_G^*(u_i,u_j)$ it follows $V(H_l)\subsetneq \mathcal{D}_{G\circ\mathcal{H}}((u_i,v_x^i),(u_j,v_y^j))$ or equivalently $|V(H_l)|< |\mathcal{D}_{G\circ\mathcal{H}}((u_i,v_x^i),(u_j,v_y^j))|$. Thus,
$$R_2=\min_{(u_i,v_x^i),(u_j,v_y^j)\in V(G\circ\mathcal{H})}\{|\mathcal{D}_{G\circ\mathcal{H}}((u_i,v_x^i),(u_j,v_y^j)))|\}>\min_{H_l\in\mathcal{H}}\{|V(H_l)|\}\ge \min_{H_l\in\mathcal{H}}\{|\mathcal{C}(H_l)|\}=\mathcal{C}(\mathcal{H}).$$
Notice that $R_2$ is strictly greater than $R_1$. So, the minimum between them is $R_1$.

Now, we assume that $u_i,u_j$ are twins, so $\mathcal{D}^*(u_i,u_j)=\emptyset$. Hence we consider two possibilities for $u_i,u_j$ in the next statements, where the conclusions are consequences of Claim \ref{claimLexi} (i) and (ii).\\
\\Subcase 2.1: If $u_i\sim u_j$, then
$|\mathcal{D}_{G\circ\mathcal{H}}((u_i,v_x^i),(u_j,v_y^j))|=|(V(H_i)-N_{H_i}(v_x^i))\cup(V(H_j)-N_{H_j}(v_y^j))|$.
So, it follows that
\begin{align*}
R_3&=\min_{(u_i,v_x^i),(u_j,v_y^j)\in V(G\circ\mathcal{H})}\{|\mathcal{D}_{G\circ\mathcal{H}}((u_i,v_x^i),(u_j,v_y^j))|\}\\
&= \min\{|(V(H_i)-N_{H_i}(v_x^i))\cup(V(H_j)-N_{H_j}(v_y^j))|\}\\
&= \min\{|V(H_i)|-\Delta(H_i)+|V(H_j)|-\Delta(H_j)\}\\
&= \min_{u_l\in V(G)}\{\mathcal{T}(u_l,\mathcal{H})\}\\
&=\mathcal{T}(G\circ\mathcal{H}).
\end{align*}
Subcase 2.2: If $d_G(u_i,u_j)=2$, then
$|\mathcal{D}_{G\circ\mathcal{H}}((u_i,v_x^i),(u_j,v_y^j))|=|N_{H_i}[v_x^i]\cup N_{H_j}[v_y^j]|$.
Similarly, we obtain that
\begin{align*}
R_4&=\min_{(u_i,v_x^i),(u_j,v_y^j)\in V(G\circ\mathcal{H})}\{|\mathcal{D}_{G\circ\mathcal{H}}((u_i,v_x^i),(u_j,v_y^j))|\} \\
&= \min\{|N_{H_i}[v_x^i]\cup N_{H_j}[v_y^j]|\}\\
&= \min\{\delta(H_i)+\delta(H_j)+2\}\\
&= \min_{u_l\in V(G)}\{\mathcal{T}(u_l,\mathcal{H})\}\\
&=\mathcal{T}(G\circ\mathcal{H}).
\end{align*}
As a conclusion of all the statements above, it is obtained that
\begin{align*}
\mathcal{D}(G\circ\mathcal{H})&=\min_{(u_i,v_x^i),(u_j,v_y^j)\in V(G\circ\mathcal{H})}\{|\mathcal{D}_{G\circ\mathcal{H}}((u_i,v_x^i),(u_j,v_y^j))|\}\\
&=\min\left\{\min_{i=j}\{|\mathcal{D}_{G\circ\mathcal{H}}((u_i,v_x^i),(u_j,v_y^j))|\},\min_{j\ne i}\{|\mathcal{D}_{G\circ\mathcal{H}}((u_i,v_x^i),(u_j,v_y^j))|\}\right\}\\
&= \min\{R_1,R_2,R_3,R_4\}\\
&= \min\{R_1,R_3,R_4\}\\
&= \min\{\mathcal{C}(\mathcal{H}),\mathcal{T}(G\circ\mathcal{H})\}.
\end{align*}
Therefore the proof is completed.
\end{proof}

Next we emphasize some particular cases of Theorem \ref{theoLexiDimensional} when the lexicographic product graphs have some specific structure which are related with the existence or not of twin vertices in the graph $G$.

\begin{corollary}
Let $G$ be a connected twins free graph of order $n\ge 2$ and let $\mathcal{H}$ be a family of $n$ non-trivial graphs. Then $G\circ\mathcal{H}$ is $\mathcal{C}(\mathcal{H})$-metric dimensional.
\end{corollary}

\begin{corollary}\label{HCondition}
Let $G$ be a connected non-trivial graph and let $H$ be a graph of order $n'\ge 2$.
\begin{enumerate}[{\rm (i)}]
\item If $G$ is twins free, then the graph $G\circ H$ is $k$-metric dimensional if and only if $k=\mathcal{C}(H)$.
\item If $G$ contains at least one false twin and one true twin, then the graph $G\circ H$ is $k$-metric dimensional if and only if $k=\min\{2\delta(H)+2,2(n'-\Delta(H)),\mathcal{C}(H)\}$.
\item If $G$ is true twins free and contains at least one false twin, then the graph $G\circ H$ is $k$-metric dimensional if and only if $k=\min\{2\delta(H)+2,\mathcal{C}(H)\}$.
\item If $G$ is false twins free and contains at least one true twin, then the graph $G\circ H$ is $k$-metric dimensional if and only if $k=\min\{2(n'-\Delta(H)),\mathcal{C}(H)\}$.
\end{enumerate}
\end{corollary}

As some instances of graphs $G$ that satisfy the conditions of the corollary above we next construct some examples. In Figure \ref{figTwinFigureH}, the vertices $v_{11}$ and $v_{12}$ of graph $G_a$ are true twins, as well as $v_{21}$ and $v_{22}$ are false twins. So, $G_a$ contains two false twins and two true twins and satisfies the premise of Corollary \ref{HCondition} (ii), and as a consequence, for any graph $H$ of order $n'\ge 2$, we have that $G_a\circ H$ is $k$-metric dimensional for $k=\min\{2\delta(H)+2,2(n'-\Delta(H)),\mathcal{C}(H)\}$. Similarly, $G_b$ is a true twins free graph and it has two false twin vertices, $v_{11}$ and $v_3$. Thus, $G_b\circ H$ is $k$-metric dimensional for $k=\min\{2\delta(H)+2,\mathcal{C}(H)\}$. Finally, the graph $G_c$ is false twins free and it has two true twin vertices, $v_{21}$ and $v_{22}$, and consequently, $G_c\circ H$ is $k$-metric dimensional for $k=\min\{2(n'-\Delta(H)),\mathcal{C}(H)\}$.

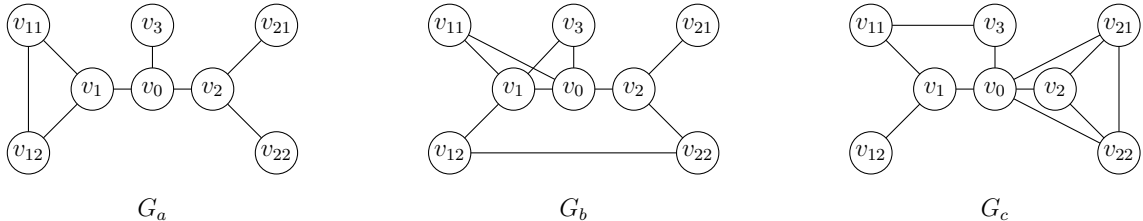
\begin{figure}[!ht]
\centering
\begin{tikzpicture}[scale=.8, transform shape, inner sep = 1pt, outer sep = 0pt, minimum size = 20pt]
\node [draw, shape=circle, fill=white] (v0) at (0,0) {$v_0$};
\node [draw, shape=circle, fill=white] (v1) at (-1,0) {$v_1$};
\node [draw, shape=circle, fill=white] (v2) at (1,0) {$v_2$};
\pgfmathparse{3*sqrt(2)/4};
\node [draw, shape=circle, fill=white] (v3) at (0,\pgfmathresult) {$v_3$};
\node [draw, shape=circle, fill=white, xshift=-1cm] (v11) at (135:1.5 cm) {$v_{11}$};
\node [draw, shape=circle, fill=white, xshift=-1cm] (v12) at (225:1.5 cm) {$v_{12}$};
\node [draw, shape=circle, fill=white, xshift=1cm] (v21) at (45:1.5 cm) {$v_{21}$};
\node [draw, shape=circle, fill=white, xshift=1cm] (v22) at (315:1.5 cm) {$v_{22}$};
\foreach \ind in {1,2,3}
\draw[black] (v0) -- (v\ind);
\draw[black] (v11) -- (v12);
\foreach \ind in {11,12}
\draw[black] (v1) -- (v\ind);
\foreach \ind in {21,22}
\draw[black] (v2) -- (v\ind);
\pgfmathparse{3*sqrt(2)/4};
\node at (0,-2) {$G_a$};
\node [draw, shape=circle, fill=white] (v0') at (7,0) {$v_0$};
\node [draw, shape=circle, fill=white] (v1') at (6,0) {$v_1$};
\node [draw, shape=circle, fill=white] (v2') at (8,0) {$v_2$};
\pgfmathparse{3*sqrt(2)/4};
\node [draw, shape=circle, fill=white] (v3') at (7,\pgfmathresult) {$v_3$};
\node [draw, shape=circle, fill=white, xshift=6cm] (v11') at (135:1.5 cm) {$v_{11}$};
\node [draw, shape=circle, fill=white, xshift=6cm] (v12') at (225:1.5 cm) {$v_{12}$};
\node [draw, shape=circle, fill=white, xshift=8cm] (v21') at (45:1.5 cm) {$v_{21}$};
\node [draw, shape=circle, fill=white, xshift=8cm] (v22') at (315:1.5 cm) {$v_{22}$};
\foreach \ind in {1',2',3',11'}
\draw[black] (v0') -- (v\ind);
\draw[black] (v12') -- (v22');
\foreach \ind in {11',12',3'}
\draw[black] (v1') -- (v\ind);
\foreach \ind in {21',22'}
\draw[black] (v2') -- (v\ind);
\pgfmathparse{3*sqrt(2)/4};
\node at (7,-2) {$G_b$};
\node [draw, shape=circle, fill=white] (v0'') at (14,0) {$v_0$};
\node [draw, shape=circle, fill=white] (v1'') at (13,0) {$v_1$};
\node [draw, shape=circle, fill=white] (v2'') at (15,0) {$v_2$};
\pgfmathparse{3*sqrt(2)/4};
\node [draw, shape=circle, fill=white] (v3'') at (14,\pgfmathresult) {$v_3$};
\node [draw, shape=circle, fill=white, xshift=13cm] (v11'') at (135:1.5 cm) {$v_{11}$};
\node [draw, shape=circle, fill=white, xshift=13cm] (v12'') at (225:1.5 cm) {$v_{12}$};
\node [draw, shape=circle, fill=white, xshift=15cm] (v21'') at (45:1.5 cm) {$v_{21}$};
\node [draw, shape=circle, fill=white, xshift=15cm] (v22'') at (315:1.5 cm) {$v_{22}$};
\foreach \ind in {1'',2'',3'',21'',22''}
\draw[black] (v0'') -- (v\ind);
\draw[black] (v21'') -- (v22'');
\draw[black] (v11'') -- (v3'');
\foreach \ind in {11'',12''}
\draw[black] (v1'') -- (v\ind);
\foreach \ind in {21'',22''}
\draw[black] (v2'') -- (v\ind);
\pgfmathparse{3*sqrt(2)/4};
\node at (14,-2) {$G_c$};
\end{tikzpicture}
\caption{The graph $G_x$, $x\in\{a,b,c\}$, satisfies the conditions of Corollary \ref{HCondition} (ii), (iii) and (iv) respectively.}
\label{figTwinFigureH}
\end{figure}

We also point out the particular case $k=2$ in Theorem \ref{theoLexiDimensional}.

\begin{corollary}
Let $G$ be a connected graph of order $n\ge 2$ and let $\mathcal{H}$ be a family of $n$ non-trivial graphs. The graph $G\circ\mathcal{H}$ is $2$-metric dimensional if and only if at least one of the following statements holds,
\begin{enumerate}[{\rm (i)}]
\item there exists $H_i\in\mathcal{H}$ which has twins or,
\item there exist two true twin vertices $u_i,u_j\in V(G)$ such that $\Delta(H_i)=n_i-1$ and $\Delta(H_j)=n_j-1$.
\item there exist two false twin vertices $u_i,u_j\in V(G)$ such that $H_i$ and $H_j$ contain at least an isolated vertex.
\end{enumerate}
\end{corollary}

It was shown in \cite{Yero2013c} a general algorithm that allow us to compute the value of $k$ for which a graph is $k$-metric dimensional. For the particular case of a graph $G\circ\mathcal{H}$, this algorithm can compute the value of $k$ in $O\left((n+\sum_{i=1}^n n_i)^3\right)$. A natural question which raises now regards with the existence of other algorithm that could allow us to compute the value of $k$ for which $G\circ\mathcal{H}$ is $k$-metric dimensional in a lower order. The next result solve precisely that fact, where the general complexity is slightly improved.

\begin{proposition}\label{prop-compute-CG}
Let $G$ be a connected graph of order $n\geq 2$ and let $\mathcal{H}$ be a family of $n$ non-trivial graphs. Then $\min\{\mathcal{T}(G\circ\mathcal{H}),\mathcal{C}(\mathcal{H})\}$ can be computed in $\displaystyle O\left(\max\left\{n^3+\sum_{u_i\in TT(G)\cup FT(G)} n_i^2,\sum_{i=1}^n n_i^3\right\}\right)$.
\end{proposition}

\begin{proof}
To compute $\mathcal{T}(G\circ\mathcal{H})$ it is first necessary to obtain the twin equivalent classes of $G$. We assume that the graph $G$ is represented by its adjacency matrix $\mathrm{AdjM}_G$. We recall that $\mathrm{AdjM}_G$ is a symmetric $(n\times n)$-matrix given by
$$\mathrm{AdjM}_G(i,j)=\left\{\begin{array}{ll}
                                1, & \mbox{if $u_i\sim u_j$}, \\
                                0, & \mbox{othewise.}
                              \end{array}\right.$$
Now, note that $u_i,u_j$ are twins if and only if for every $u_r\in V(G)-\{u_i,u_j\}$, we have that $\mathrm{AdjM}_G(i,r)=\mathrm{AdjM}_G(j,r)$. Given two twin vertices $u_i,u_j$, if $\mathrm{AdjM}_G(i,j)=1$, then $u_i,u_j$ are true twins, otherwise they are false twins. Note that determining if two vertices are twins can be checked in linear time. In the worst case, when all twin equivalent classes are singletons, it would be necessary to check for any pair of vertices if they are twins or not. Thus, we conclude that determining the twin equivalent classes of $G$ can be computed in $O(n^3)$. Once determined the twin equivalent classes of $G$, we have the following three possibilities for each twin equivalent class $U_G$ of $G$.
\begin{itemize}
\item If $U_G=\{u_i\}$ is a singleton vertex, then we take the order $n_i$ of $H_i$, as the representative value of this class.
\item If $U_G$ is a false twin equivalence class, then we take $\displaystyle\min_{u_j,u_l\in U_G}\{\delta(H_j)+\delta(H_l)+2\}$ as the representative value of this class.
\item If $U_G$ is a true twin equivalence class, then we take $\displaystyle\min_{u_j,u_l\in U_G}\{|V(H_j)|-\Delta(H_j)+|V(H_l)|-\Delta(H_l)\}$ as the representative value of this class.
\end{itemize}
We observe that $\mathcal{T}(G\circ\mathcal{H})$ is the minimum of the representative values of each twin equivalence class. The minimum and maximum degrees $\delta(H_i)$ and $\Delta(H_i)$ of the graphs $H_i$ (of order $n_i$) can be computed in $O(n_i^2)$. So, computing the representative value of each non-singleton twin equivalence class $U_G$ can be done in $O(\sum_{u_i\in U_G} n_i^2)$. Therefore, we can compute the value of $\mathcal{T}(G\circ\mathcal{H})$ in $O(n^3+\sum_{u_i\in TT(G)\cup TF(G)} n_i^2)$.

On the other hand, it was shown in \cite{Estrada-Moreno2014a}, that we can compute $\mathcal{C}(H_i)$ in $O(n_i^3)$ for every $H_i\in\mathcal{C}(\mathcal{H})$. Thus, the complete computation of $\mathcal{C}(\mathcal{H})$ can be performed in $O(\sum_{i=1}^n n_i^3)$, which complete the proof.
\end{proof}

\section{The $k$-metric dimension of $G\circ\mathcal{H}$}\label{sect-k-dim}

Once studied the suitable values of $k$ for which a given graph is $k$-metric dimensional, in this section we focus into obtaining the $k$-metric dimension of the lexicographic product of graphs for these values of $k$, \textit{i.e.}, $k\in\{1,\ldots,\min\{\mathcal{T}(G\circ\mathcal{H}),\mathcal{C}(\mathcal{H})\}\}$.

Note that a trivial upper bound on the $k$-metric dimension of $G\circ\mathcal{H}$ is $|V(G\circ\mathcal{H})|$, which is tight at least for $k=2$. To see this, we can firstly refer to a result shown in \cite{Estrada-Moreno2013}, which states that the $2$-metric dimension of a graph $G$ is equal to its order if and only if $G$ has no singleton twin equivalence classes. Considering this fact, we can conclude the next result.

\begin{remark}
Let $G$ be a connected graph of order $n\ge 2$ and let $\mathcal{H}=\{H_1,\ldots,H_n\}$ be a family of non-trivial graphs. Then $\dim_2(G\circ\mathcal{H})=|V(G\circ\mathcal{H})|$ if and only if the following statements hold.
\begin{enumerate}[{\rm (i)}]
\item For every $u_i\in S(G)$, the graph $H_i\in\mathcal{H}$ has no singleton twin equivalence classes.
\item For every $u_i\in TT(G)$, either the graph $H_i\in\mathcal{H}$ has no singleton twin equivalence classes or $H_i$ has exactly one singleton twin equivalence class $\{v_i\}$, where $\delta(v_i)=n_i-1$, and there exists $u_j\in TT(u_i)$ such that $H_j\in\mathcal{H}$ has a vertex $v_j$ of degree $\delta(v_j)=n_j-1$.
\item For every $u_i\in FT(G)$, either the graph $H_i\in\mathcal{H}$ has no singleton twin equivalence classes or $H_i$ has exactly one singleton twin equivalence class $\{v_i\}$, where $\delta(v_i)=0$, and there exists $u_j\in FT(u_i)$ such that $H_j\in\mathcal{H}$ has a vertex $v_j$ of degree $\delta(v_j)=0$.
\end{enumerate}
\end{remark}

On the other hand, now we give a lower bound for $\dim_k(G\circ\mathcal{H})$, in terms of $\adim_k(H_i)$ for every $H_i\in \mathcal{H}$, which is also tight.

\begin{theorem}\label{lowerBoundLexiAdj}
Let $G$ be a connected graph of order $n\ge 2$ and let $\mathcal{H}=\{H_1,\ldots,H_n\}$ be a family of non-trivial graphs. For any $k\in\{1,\ldots,\min\{\mathcal{T}(G\circ\mathcal{H}),\mathcal{C}(\mathcal{H})\}\}$, $$\dim_k(G\circ\mathcal{H})\ge\sum_{i=1}^n\adim_k(H_i).$$
\end{theorem}

\begin{proof}
Let $S$ be a $k$-metric basis for $G\circ\mathcal{H}$ and let $S_i=\{v_j^i:(u_i,v_j^i)\in S\}$. By Claim \ref{claimLexi} (i), we deduce that for any $(u_i,v_j^i),(u_i,v_l^i)\in \{u_i\}\times V(H_i)$ it holds that $|\mathcal{D}_{G\circ\mathcal{H}}((u_i,v_j^i),(u_i,v_l^i))\cap (\{u_i\}\times S_i)|\ge k$. Also, by Claim \ref{claimLexi} (ii), $\mathcal{D}_{G\circ\mathcal{H}}((u_i,v_j^i),(u_i,v_l^i))=\{u_i\}\times\mathcal{C}_{H_i}(v_j^i,v_l^i)$ and, as a consequence, $|S_i\cap \mathcal{C}_{H_i}(v_j^i,v_l^i)|\ge k$. Thus, $S_i$ is a $k$-adjacency generator for $H_i$ and we obtain that $|S_i|\ge\adim_k(H_i)$. Therefore, $\dim_k(G\circ\mathcal{H})=|S|=\sum_{i=1}^n |S_i|\ge\sum_{i=1}^n\adim_k(H_i)$.
\end{proof}

Later on, in Theorem \ref{AdjEqual}, we show the tightness of the result above. To this end we need some extra notation. Given a graph $G$ with vertex set $V(G)=\{u_1,u_2,\ldots,u_n\}$ and a  family of graphs $\mathcal{H}=\{H_1,\ldots,H_n\}$, we define the following properties on the triplet $(G,\mathcal{H},k)$.\\

\textbf{Property $\mathcal{P}_1$:}
For any $u_i\in TT(G)$, where $TT(u_i)=\{u_{i_1},u_{i_2},\ldots,u_{i_r}\}$, there exist $i_r$ $k$-adjacency bases $A_{i_1}^t,A_{i_2}^t,\ldots,A_{i_r}^t$ of $H_{i_1} ,H_{i_2},\ldots,H_{i_r}$, respectively, such that for every $j,l\in\{1,\ldots,r\}$, $j\ne l$, and every $x\in V(H_{i_j})$ and $y\in V(H_{i_l})$ it follows, $$|(A_{i_j}^t\cap(V(H_{i_j})-N_{H_{i_j}}(x)))\cup(A_{i_l}^t\cap(V(H_{i_l})-N_{H_{i_l}}(y)))|\ge k.$$

Notice that Property $\mathcal{P}_1$ ensures that for any $u_i,u_j\in TT(G)$, $i\ne j$, there exist two $k$-adjacency bases $A_i^t$, $A_j^t$ of $H_i$, $H_j$, respectively, such that vertices belonging to $\{u_i\}\times H_i$ are distinguished from vertices belonging to $\{u_j\}\times H_j$ by at least $k$ vertices of $(\{u_i\}\times A_i^t)\cup (\{u_j\}\times A_j^t)$.

An example which helps to clarify the above property is, for instance, the triplet $(K_3,\mathcal{H},2)$, where $V(K_3)=\{u_1,u_2,u_3\}$ and $\mathcal{H}=\{C_5^1,C_5^2,C_5^3\}$. Figure \ref{figK_3-C_5} shows the family of graphs $\mathcal{H}$. In this case $TT(u_1)=\{u_1,u_2,u_3\}=TT(K_3)$, since there is only one true twin equivalence class. If we take as $2$-adjacency bases $A_{1_1}=\{v_1^1, v_3^1,v_4^1\}$, $A_{1_2}=\{v_1^2, v_3^2,v_4^2\}$ and $A_{1_3}=\{v_1^3, v_3^3,v_4^3\}$ of $C_5^1,C_5^2$ and $C_5^3$, respectively, then $(K_3,C_5,2)$ satisfies Property $\mathcal{P}_1$. For instance, if $x=v_2^1$ and $y=v_5^3$, then $|(A_{1_1}\cap(V(C_5^1)-N_{C_5^1}(v_2^1)))\cup(A_{1_3}\cap(V(C_5^3)-N_{C_5^3}(v_5^3)))|= |(A_{1_1}\cap\{v_2^1,v_4^1,v_5^1\})\cup(A_{1_3}\cap\{v_2^3,v_3^3,v_5^3\})|=|\{v_4^1\}\cup\{v_3^3\}|=2\ge 2$.\\

\begin{figure}[!ht]
\centering
\begin{tikzpicture}[scale=.9, transform shape, inner sep = 1pt, outer sep = 0pt, minimum size = 15pt]
\foreach \ind in {1,...,5}
\pgfmathparse{-36+360/5*\ind}
\node [draw, xshift=-3 cm, shape=circle, fill=white] (a\ind) at (\pgfmathresult:1) {$v^1_\ind$};
\node [draw, xshift=-3 cm, shape=circle, fill=white, line width=1.5pt] at (36:1) {$v^1_1$};
\node [draw, xshift=-3 cm, shape=circle, fill=white, line width=1.5pt] at (178:1) {$v^1_3$};
\node [draw, xshift=-3 cm, shape=circle, fill=white, line width=1.5pt] at (250:1) {$v^1_4$};
\foreach \ind in {1,...,5}
\pgfmathparse{360/5*\ind}
\node [draw, xshift=3 cm, shape=circle, fill=white] (b\ind) at (\pgfmathresult:1) {$v^2_\ind$};
\node [draw, xshift=3 cm, shape=circle, fill=white, line width=1.5pt] at (72:1) {$v^2_1$};
\node [draw, xshift=3 cm, shape=circle, fill=white, line width=1.5pt] at (216:1) {$v^2_3$};
\node [draw, xshift=3 cm, shape=circle, fill=white, line width=1.5pt] at (288:1) {$v^2_4$};
\foreach \ind in {1,...,5}
\pgfmathparse{18+360/5*\ind}
\node [draw, yshift=4.196 cm, shape=circle, fill=white] (c\ind) at (\pgfmathresult:1) {$v^3_\ind$};
\node [draw, yshift=4.196 cm, shape=circle, fill=white, line width=1.5pt] at (90:1) {$v^3_1$};
\node [draw, yshift=4.196 cm, shape=circle, fill=white, line width=1.5pt] at (234:1) {$v^3_3$};
\node [draw, yshift=4.196 cm, shape=circle, fill=white, line width=1.5pt] at (306:1) {$v^3_4$};
\foreach \ind in {1,...,4}
\pgfmathparse{int(\ind+1)}
\draw (a\ind) -- (a\pgfmathresult);
\draw (a5) -- (a1);
\foreach \ind in {1,...,4}
\pgfmathparse{int(\ind+1)}
\draw (b\ind) -- (b\pgfmathresult);
\draw (b5) -- (b1);
\foreach \ind in {1,...,4}
\pgfmathparse{int(\ind+1)}
\draw (c\ind) -- (c\pgfmathresult);
\draw (c5) -- (c1);
\draw [dashed] (-1.2,.6) -- (1.2,.6);
\draw [dashed] (-1.2,-.6) -- (1.2,-.6);
\draw [dashed] (-1,-.4) -- (1,.4);
\draw [dashed] (1,-.4) -- (-1,.4);

\draw [dashed] (.8,2.7) -- (2,1);
\draw [dashed] (1.3647,4) -- (3.2,1.4);
\draw [dashed] (1,2.7) -- (3,1.4);
\draw [dashed] (1.3647,3.6) -- (2,1.2);

\draw [dashed] (-.8,2.7) -- (-2,1);
\draw [dashed] (-1.3647,4) -- (-3.2,1.4);
\draw [dashed] (-1,2.7) -- (-3,1.4);
\draw [dashed] (-1.3647,3.6) -- (-2,1.2);

\end{tikzpicture}
\caption{Sketch of lexicographic product $K_3\circ C_5$, where the dashed line between two cycles $C_5$ means that each vertex of a cycle is connected to all vertices of the other cycle. The vertices represented by thick lines form a $2$-adjacency basis of each copy of $C_5$.}
\label{figK_3-C_5}
\end{figure}
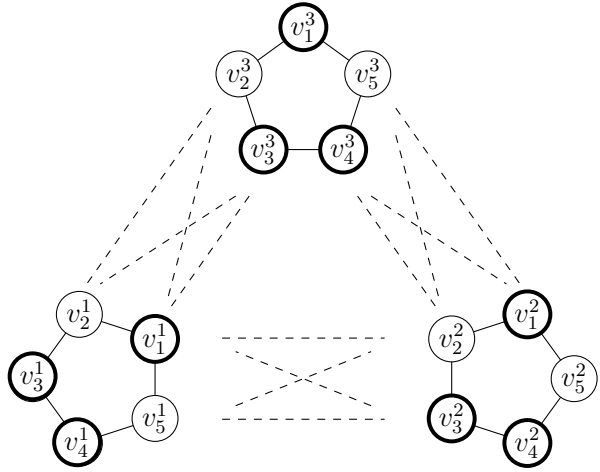

\textbf{Property $\mathcal{P}_2$:}
For any $u_i\in FT(G)$, where $FT(u_i)=\{u_{i_1},u_{i_2},\ldots,u_{i_r}\}$, there exist $i_r$ $k$-adjacency bases $A_{i_1}^f,A_{i_2}^f,\ldots,A_{i_r}^f$ of $H_{i_1} ,H_{i_2},\ldots,H_{i_r}$, respectively, such that for every $j,l\in\{1,\ldots,r\}$, $j\ne l$, and every $x\in V(H_{i_j})$ and $y\in V(H_{i_l})$ it follows, $$|(A_{i_j}^f\cap N_{H_{i_j}}[x])\cup(A_{i_l}^f\cap N_{H_{i_l}}[y])|\ge k.$$

Notice that Property $\mathcal{P}_2$ ensures that for any $u_i,u_j\in FT(G)$, $i\ne j$, there exist two $k$-adjacency bases $A_i^f$, $A_j^f$ of $H_i$, $H_j$, respectively, such that vertices belonging to $\{u_i\}\times H_i$ are distinguished from vertices belonging to $\{u_j\}\times H_j$ by at least $k$ vertices of $\left(\{u_i\}\times A_i^f\right)\cup \left(\{u_j\}\times A_j^f\right)$.

Further on we will see a triplet $(G,\mathcal{H},k)$ that satisfy Properties $\mathcal{P}_1$ and $\mathcal{P}_2$ at the same time, when $\mathcal{H}$ is a family of paths of order greater than three and/or cycles of order greater than four, $G$ is any non-trivial connected graph and $k\in\{2,3\}$.

To continue our exposition we need some extra notation. Given a family of graphs $\mathcal{H}=\{H_1,\ldots,H_n\}$, we define $\mathcal{\overline{H}}$ as the family of complement graphs of each $H_i\in\mathcal{H}$, \textit{i.e.}, $\mathcal{\overline{H}}=\{\overline{H}_1,\ldots,\overline{H}_n\}$.
We see now how Properties $\mathcal{P}_1$ and $\mathcal{P}_2$ behave for the case of the triplet $(G,\overline{\mathcal{H}},k)$. To this end, we need to define two other properties on the triplet $(G,\mathcal{H},k)$.\\

\textbf{Property $\mathcal{P}_3$:}
For any $u_i\in TT(G)$, where $TT(u_i)=\{u_{i_1},u_{i_2},\ldots,u_{i_r}\}$, there exist $i_r$ $k$-adjacency bases $A_{i_1}^t,A_{i_2}^t,\ldots,A_{i_r}^t$ of $H_{i_1} ,H_{i_2},\ldots,H_{i_r}$, respectively, such that for every $j,l\in\{1,\ldots,r\}$, $j\ne l$, and every $x\in V(H_{i_j})$ and $y\in V(H_{i_l})$, it follows $$|(A_{i_j}^t\cap N_{H_{i_j}}(x))\cup(A_{i_l}^t\cap N_{H_{i_l}}(y))|\ge k.$$

\textbf{Property $\mathcal{P}_4$:}
For any $u_i\in FT(G)$, where $FT(u_i)=\{u_{i_1},u_{i_2},\ldots,u_{i_r}\}$, there exist $i_r$ $k$-adjacency bases $A_{i_1}^f,A_{i_2}^f,\ldots,A_{i_r}^f$ of $H_{i_1} ,H_{i_2},\ldots,H_{i_r}$, respectively, such that for every $j,l\in\{1,\ldots,r\}$, $j\ne l$, and every $x\in V(H_{i_j})$ and $y\in V(H_{i_l})$ it follows, $$|(A_{i_j}^f\cap(V(H_{i_j})-N_{H_{i_j}}(x)))\cup(A_{i_l}^f\cap(V(H_{i_l})-N_{H_{i_l}}(y)))|\ge k.$$

Next claim relates all the above properties while using them in $(G,\mathcal{H},k)$ or $(G,\overline{\mathcal{H}},k)$.

\begin{claim}\label{ClaimP3-P4}
Let $G$ be a graph with vertex set $V(G)=\{u_1,u_2,\ldots,u_n\}$ and let $\mathcal{H}=\{H_1,\ldots,H_n\}$ be a family of $n$ graphs. Then,
\begin{enumerate}[{\rm (i)}]
\item the triplet $(G,\overline{\mathcal{H}},k)$ satisfies Property $\mathcal{P}_1$ if and only if $(G,\mathcal{H},k)$ satisfies Property $\mathcal{P}_3$,
\item the triplet $(G,\overline{\mathcal{H}},k)$ satisfies Property $\mathcal{P}_2$ if and only if $(G,\mathcal{H},k)$ satisfies Property $\mathcal{P}_4$.
\end{enumerate}
\end{claim}

\begin{proof}
For any $u_i\in TT(G)$ and any $v_j\in V(H_i)$ we have that $A_i^t\cap N_{H_i}(v_j)=A_i^t\cap (V(\overline{H}_i)-N_{\overline{H}_i}(v_j))$. Thus, the set of $k$-adjacency bases $\{A_{i_1}^t,A_{i_2}^t,\ldots,A_{i_r}^t\}$ which makes that the triplet $(G,\mathcal{\overline{H}},k)$ satisfies Property $\mathcal{P}_1$, also makes that $(G,\mathcal{H},k)$ satisfies Property $\mathcal{P}_3$ and vice versa. Therefore (i) follows. The item (ii) follows similarly from the fact that for any $u_i\in FT(G)$ and any $v_j\in V(H_i)$ we have that $A_i^f\cap (V(H_i)-N_{H_i}(v_i))= A_i^f\cap N_{\overline{H}_i}[v_i]$.
\end{proof}

In this point we are able to give one of the main results of this work and its powerful consequences.

\begin{theorem}\label{AdjEqual}
Let $G$ be a connected graph of order $n\ge 2$, let $\mathcal{H}$ be a family of $n$ non-trivial graphs and let $k\in\{1,\ldots,\min\{\mathcal{T}(G\circ\mathcal{H}),\mathcal{C}(\mathcal{H})\}\}$.
\begin{enumerate}[{\rm (i)}]
\item If the triplet $(G,\mathcal{H},k)$ satisfies Properties $\mathcal{P}_1$ and $\mathcal{P}_2$, then $$\dim_k(G\circ\mathcal{H})=\sum_{i=1}^n\adim_k(H_i).$$
\item If the triplet $(G,\mathcal{H},k)$ satisfies Properties $\mathcal{P}_3$ and $\mathcal{P}_4$, then $$\dim_k(G\circ\mathcal{\overline{H}})=\sum_{i=1}^n\adim_k(H_i).$$
\end{enumerate}
\end{theorem}

\begin{proof}
(i) We assume that the triplet $(G,\mathcal{H},k)$ satisfies Properties $\mathcal{P}_1$ and $\mathcal{P}_2$. If $u_i\in TT(G)$ or $u_i\in FT(G)$, then we take a $k$-adjacency basis $A_i^t$  or $A_i^f$ of $H_i$ as defined in Property $\mathcal{P}_1$ or Property $\mathcal{P}_2$, respectively. Also, if $u_i\in V(G)$ is not a twin vertex, then we take any $k$-adjacency basis $A_i$ of $H_i$.

We claim that $$B=\left( \bigcup_{u_i\in TT(G)}\{u_i\}\times A_i^t \right)\cup \left( \bigcup_{u_i\in FT(G)}\{ u_i\}\times A_i^f \right)\cup  \left( \bigcup_{u_i\not\in TT(G)\cup FT(G)}\{ u_i\}\times A_i  \right) .$$    is a $k$-metric generator for $G\circ\mathcal{H}$.

We differentiate the following four cases for  two different vertices  $(u_i,v_r^i),(u_j,v_s^j)\in V(G\circ\mathcal{H})$.\\
\\Case 1. $i=j$. In this case $r\ne s$. We have three possibilities for the vertex $u_i$
\begin{itemize}
\item $u_i\in TT(G)$, in which case   $B\cap (\{u_i\}\times V(H_i))=\{u_i\}\times A_i^t$,
\item $u_i\in FT(G)$, in which case $B\cap (\{u_i\}\times V(H_i))=\{u_i\}\times A_i^f$,
\item $u_i\not \in TT(G)\cup FT(G)$, in which case   $B\cap (\{u_i\}\times V(H_i))=\{u_i\}\times A_i $.
\end{itemize}

Since $A_i^t$, $A_i^f$ and $A_i$ are  $k$-adjacency bases of $H_i$, we obtain that $|\mathcal{C}_{H_i}(v_r^i,v_s^i)\cap A_i^t|\ge k$, $|\mathcal{C}_{H_i}(v_r^i,v_s^i)\cap A_i^f|\ge k$ and $|\mathcal{C}_{H_i}(v_r^i,v_s^i)\cap A_i|\ge k$.
In any case, as $\mathcal{D}_{G\circ\mathcal{H}}((u_i,v_r^i),(u_i,v_s^i))=\{u_i\}\times\mathcal{C}_{H_i}(v_r^i,v_s^i)$, we conclude that $|B\cap \mathcal{D}_{G\circ\mathcal{H}}((u_i,v_r^i),(u_i,v_s^i))|\ge k$.\\
\\Case 2. $i\ne j$ and $u_i,u_j$ are true twins. So $\mathcal{D}_{G\circ\mathcal{H}}((u_i,v_r^i),(u_j,v_s^j))=(V(H_i)-N_{H_i}(v_r^i))\cup(V(H_j)-N_{H_j}(v_s^j))$. Since $(G,\mathcal{H},k)$ satisfies Property $\mathcal{P}_1$, there exist at least $k$ elements of $(\{u_i\}\times A_i^t)\cup (\{u_j\}\times A_j^t)\subseteq B$ distinguishing $(u_i,v_r^i),(u_j,v_s^j)$.\\
\\Case 3. $i\ne j$ and $u_i,u_j$ are false twins. Thus $\mathcal{D}_{G\circ\mathcal{H}}((u_i,v_r^i),(u_j,v_s^j))=N_{H_i}[v_r^i]\cup N_{H_j}[v_s^j]$. Since $(G,\mathcal{H},k)$ satisfies Property $\mathcal{P}_2$, there exist at least $k$ elements of $(\{u_i\}\times A_i^f)\cup (\{u_j\}\times A_j^f)\subseteq B$ distinguishing $(u_i,v_r^i),(u_j,v_s^j)$.\\
\\Case 4. $i\ne j$ and $u_i,u_j$ are not twins. Hence, there exists $u_l\in\mathcal{D}_G^*(u_i,u_j)$. So, the  set $B\cap(\{u_l\}\times V(H_l))$  is either $\{u_l\}\times A_l^t$ or $\{u_l\}\times A_l^f$ or $\{u_l\}\times A_l$. Hence,
 $(u_i,v_r^i)$ and $(u_j,v_s^j)$ are
distinguished by, at least $k$ elements of $B$.

Therefore, $B$ is a $k$-metric generator for $G\circ\mathcal{H}$, and consequently, $\dim_k(G\circ\mathcal{H})\le |B|=\sum_{i=1}^n\adim_k(H_i)$. By Theorem \ref{lowerBoundLexiAdj}, we conclude the proof of (i).\\

\noindent (ii) We assume that the triplet $(G,\mathcal{H},k)$ satisfies Properties $\mathcal{P}_3$ and $\mathcal{P}_4$. By Claim \ref{ClaimP3-P4}, the triplet $(G,\mathcal{\overline{H}},k)$ satisfies Properties $\mathcal{P}_1$ and $\mathcal{P}_2$. Proceeding analogously to the proof of (i) and given that $\adim_k(\overline{H}_i)=\adim_k(H_i)$ for every $H_i\in\mathcal{H}$, we have that $\dim_k(G\circ\mathcal{\overline{H}})=\sum_{i=1}^n\adim_k(\overline{H}_i)=\sum_{i=1}^n\adim_k(H_i)$.
\end{proof}

The theorem above is a generalization for $k\in\{1,\ldots,\mathcal{C}(\mathcal{H})\}$ of a result obtained by Jannesari and Omoomi \cite{JanOmo2012} for the metric dimension of $G\circ H$, \textit{i.e}, for $\dim_k(G\circ\mathcal{H})$ when $k=1$ and the graphs belonging to $\mathcal{H}$ are isomorphic to the same graph $H$.

Assume now that the $k$-adjacency dimension of every graph of a given family $\mathcal{H}'$ is known. Hence, as a measure of the reach of Theorem \ref{AdjEqual}, the following consequences are deduced. Notice that we can then compute, not only the $k$-metric dimension of $G\circ\mathcal{H}'$, but also that of $G\circ\mathcal{\overline{H}'}$, for a huge quantity of graphs $G$. If $G$ is a connected graph of order $n\ge 2$ and $\mathcal{H}$ is a family of $n$ non-trivial graphs, then Theorem \ref{AdjEqual} gives us the conditions for which the problem of computing the $k$-metric dimension of $G\circ\mathcal{H}$ and $G\circ\mathcal{\overline{H}}$ is reduced to computing the $k$-adjacency dimension of the graphs $H_i\in \mathcal{H}$.

\begin{corollary}\label{AllComb_TT_FT}
Let $G$ be a connected graph of order $n\ge 2$, let $\mathcal{H}$ be a family of $n$ non-trivial graphs and let $k\in\{1,\ldots,\min\{\mathcal{T}(G\circ\mathcal{H}),\mathcal{C}(\mathcal{H})\}\}$. Then the following statements hold.
\begin{enumerate}[{\rm (i)}]
\item If $G$ is twins free, then $$\dim_k(G\circ\mathcal{H})=\dim_k(G\circ\mathcal{\overline{H}})=\sum_{i=1}^n\adim_k(H_i).$$
\item If $G$ is false twins free and $(G,\mathcal{H},k)$ holds Property $\mathcal{P}_1$, $$\dim_k(G\circ\mathcal{H})=\sum_{i=1}^n\adim_k(H_i).$$
\item If $G$ is true twins free  and $(G,\mathcal{H},k)$ holds Property $\mathcal{P}_2$, $$\dim_k(G\circ\mathcal{H})=\sum_{i=1}^n\adim_k(H_i).$$
\end{enumerate}
\end{corollary}

A natural question is now the following one. Can we realize triplets $(G,\mathcal{H},k)$ satisfying Properties $\mathcal{P}_1$, $\mathcal{P}_2$, $\mathcal{P}_3$ or $\mathcal{P}_4$? To proceed in this direction, we first need to present some useful lemmas which allow us to describe some realizations of the triplet $(G,\mathcal{H},k)$ in concordance with Properties $\mathcal{P}_1$ and $\mathcal{P}_2$.

\begin{lemma}\label{DeltaTrue}
Let $G$ be a connected graph of order $n\ge 2$ and let $\mathcal{H}$ be a family of $n$ non-trivial graphs. If $\adim_k(H_i)-\Delta(H_i)\ge\lceil\frac{k}{2}\rceil$ for every $H_i\in\mathcal{H}$ and $k\in\{1,\ldots,\min\{\mathcal{T}(G\circ\mathcal{H}),\mathcal{C}(\mathcal{H})\}\}$, then $(G,\mathcal{H},k)$ satisfies Properties $\mathcal{P}_1$ and $\mathcal{P}_4$.
\end{lemma}

\begin{proof}
Let $A_i,A_j$, $i\ne j$, be two $k$-adjacency bases of $H_i,H_j\in\mathcal{H}$, respectively. Since $\adim_k(H_i)-\Delta(H_i)\ge\lceil\frac{k}{2}\rceil$ and $\adim_k(H_j)-\Delta(H_j)\ge\lceil\frac{k}{2}\rceil$, it follows that for every $v\in V(H_i)$ and $w\in V(H_j)$, $|A_i-N_{H_i}(v)|\ge\lceil\frac{k}{2}\rceil$ and $|A_j-N_{H_j}(w)|\ge\lceil\frac{k}{2}\rceil$. Thus, we deduce $|A_i\cap(V(H_i)-N_{H_i}(v))|\ge\lceil\frac{k}{2}\rceil$ and $|A_j\cap(V(H_j)-N_{H_j}(w))|\ge\lceil\frac{k}{2}\rceil$, which implies that $(G,\mathcal{H},k)$ satisfies Properties $\mathcal{P}_1$ and $\mathcal{P}_4$.
\end{proof}

\begin{lemma}\label{DeltaFalse}
Let $G$ be a connected graph of order $n\ge 2$ and let $\mathcal{H}$ be a family of $n$ graphs without isolated vertices. If $\Delta(H_i)-1\le\lfloor\frac{k}{2}\rfloor$ for every $H_i\in\mathcal{H}$ and $k\in\{1,\ldots,\min\{\mathcal{T}(G\circ\mathcal{H}),\mathcal{C}(\mathcal{H})\}\}$, then $(G,\mathcal{H},k)$ satisfies Property $\mathcal{P}_2$.
\end{lemma}

\begin{proof}
Let $v_x^i\in V(H_i)$ and let $A_i$ be $k$-adjacency basis of $H_i$. Since $N_{H_i}(v_x^i)\ne\emptyset$, for every $v_y^i\in N_{H_i}(v_x^i)$ we have that $|N_{H_i}(v_y^i)-\{v_x^i\}|\le\Delta(H_i)-1\le\lfloor\frac{k}{2}\rfloor$. Now, as $|A_i\cap\mathcal{C}(v_x^i,v_y^i)|\ge k$, we obtain $|A_i\cap N_{H_i}[v_x^i]|\ge\lceil\frac{k}{2}\rceil$. Thus, for every $H_l,H_j\in\mathcal{H}$, $l\ne j$, and every $v_x^l\in V(H_l)$, $v_y^j\in V(H_j)$ it follows that $|(A_l\cap N_{H_l}[v_x^l])\cup(A_j\cap N_{H_j}[v_y^j])|\ge k$, where $A_l,A_j$ are $k$-adjacency bases of $H_l,H_j$, respectively. Therefore, $(G,\mathcal{H},k)$ satisfies Property $\mathcal{P}_2$.
\end{proof}

According to the lemmas above, we can notice now that, for instance, any triplet $(G,\mathcal{H},k)$, where $G$ is any connected graph, $\mathcal{H}$ is formed by paths of order greater than three and/or cycles of order greater than five, and $k\in \{2,3\}$ (or if $\mathcal{H}$ is only formed by cycles, then also happens for $k=4$), satisfies Properties $\mathcal{P}_1$ and $\mathcal{P}_2$. In this sense, by Theorem \ref{AdjEqual}, the previous lemmas and Proposition \ref{value-adj-Paths-Cycles}, we give a closed formulae for the lexicographic product of any graph $G$ and this family $\mathcal{H}$ of graphs.

\begin{theorem}\label{formulaePathsCycles}
Let $G$ be a connected graph of order $n\ge 2$ and let $\mathcal{H}=\{P_{q_1},\ldots,P_{q_r},C_{q_{r+1}},\ldots,C_{q_n}\}$. If $q_i\ge 4$ for $1\le i\le r$ and $q_i\ge 5$ for $r+1\le i\le n$, then
\begin{enumerate}[{\rm (i)}]
\item $\displaystyle\dim_2(G\circ\mathcal{H})=\sum_{i=1}^r \left\lceil\frac{q_i+1}{2}\right\rceil+\sum_{i=r+1}^n \left\lceil\frac{q_i}{2}\right\rceil$
\item $\displaystyle\dim_3(G\circ\mathcal{H})=\sum_{i=1}^r\left(q_i-\left\lfloor\frac{q_i-4}{5}\right\rfloor\right)+\sum_{i=r+1}^n \left(q_i-\left\lfloor\frac{q_i}{5}\right\rfloor\right)$.
\end{enumerate}
Moreover, if $\mathcal{H}=\{C_{q_{1}},\ldots,C_{q_n}\}$ and $q_i\ge 5$, then $\displaystyle\dim_4(G\circ\mathcal{H})=\sum_{i=1}^n q_i$.
\end{theorem}

From Lemma \ref{DeltaTrue}, we deduce that any triplet $(G,\mathcal{H},k)$, where $G$ is any connected graph, $\mathcal{H}$ is formed by paths of order greater than three and/or cycles of order greater than five and $k\in \{2,3\}$ (or if $\mathcal{H}$ is only formed by cycles, then also happens for $k=4$), satisfies Property $\mathcal{P}_4$. However Lemma \ref{DeltaFalse} does not ensure that $(G,\mathcal{H},k)$ satisfies Property $\mathcal{P}_3$ for $k\in \{2,3\}$ and a graph $G$ that contains at least a true twin. From Theorem \ref{AdjEqual} and Proposition \ref{value-adj-Paths-Cycles}, we can conclude the following result.

\begin{theorem}\label{formulaePathsCyclesComplement}
Let $G$ be a connected true twins free graph of order $n\ge 2$ and let $\mathcal{H}=\{P_{q_1},\ldots,P_{q_r},$ $C_{q_{r+1}},\ldots,C_{q_n}\}$. If $q_i\ge 4$ for $1\le i\le r$ and $q_i\ge 5$ for $r+1\le i\le n$, then
\begin{enumerate}[{\rm (i)}]
\item $\displaystyle\dim_2(G\circ\overline{\mathcal{H}})=\sum_{i=1}^r \left\lceil\frac{q_i+1}{2}\right\rceil+\sum_{i=r+1}^n \left\lceil\frac{q_i}{2}\right\rceil$
\item $\displaystyle\dim_3(G\circ\overline{\mathcal{H}})=\sum_{i=1}^r\left(q_i-\left\lfloor\frac{q_i-4}{5}\right\rfloor\right)+\sum_{i=r+1}^n \left(q_i-\left\lfloor\frac{q_i}{5}\right\rfloor\right)$.
\end{enumerate}
Moreover, if $G$ is a connected graph of order $n\ge 2$, $\mathcal{H}=\{C_{q_{1}},\ldots,C_{q_n}\}$ and $q_i\ge 5$, then $\displaystyle\dim_4(G\circ\overline{\mathcal{H}})=\sum_{i=1}^n q_i$.
\end{theorem}

Note that for any connected graph $G$ of order $n$ and any family $\mathcal{H}$ of $n$ graphs we have that $\dim_4(G\circ\mathcal{H})=\dim_4(G\circ\overline{\mathcal{H}})=\sum_{i=1}^n q_i=|V(G\circ\mathcal{H})|$, and these are two other examples where the trivial upper bound is reached.

To finish this section, we continue now with some examples of classes of graphs achieving the equality in the bound of Theorem \ref{lowerBoundLexiAdj}. To this end, we need the following results from \cite{Estrada-Moreno2013corona} regarding the join graph $K_1+G$. We recall that the \emph{join graph} $G+H$ of the graphs $G=(V_{1},E_{1})$ and $H=(V_{2},E_{2})$ is the graph with vertex set $V(G+H)=V_{1}\cup V_{2}$ and edge set $E(G+H)=E_{1}\cup E_{2}\cup \{uv\,:\,u\in V_{1},v\in V_{2}\}$.

\begin{lemma}\label{lemmaNoK1}{\rm \cite{Estrada-Moreno2013corona}}
Let $G$ be a connected graph. If $D(G)\ge 6$, or $G$ is a cycle of order at least seven, or $G$ is a path of order at least six, then the vertex of $K_1$ does not belong to any $k$-metric basis of $K_1+G$ for any $k\in\{1,\ldots,\mathcal{C}(G)\}$.
\end{lemma}

As another necessary tool, we also give a relationship between $\dim_k(K_1+G)$, $\adim_k(G)$ and $\adim(\overline{G})$ for a graph $G$ with diameter at least six.

\begin{lemma}\label{relationDim-AdimK_1}
Let $G$ be a connected graph. If $D(G)\ge 6$, then $\dim_k(K_1+G)=\adim_k(G)=\adim(\overline{G})$ for $k\in\{1,\ldots,\mathcal{C}(G)\}$.
\end{lemma}

\begin{proof}
We assume that $D(G)\ge 6$. Let $u$ be the vertex of $K_1$ in $K_1+G$. First we note that, if $x,y\in V(G)$, then $\mathcal{C}_G(x,y)=\mathcal{D}_{K_1+G}(x,y)$ and $u\notin\mathcal{D}_{K_1+G}(x,y)$. Now, if $B$ is a $k$-metric basis of $K_1+G$, then for every $x,y\in V(G)$ it follows that $|B\cap \mathcal{D}_{K_1+G}(x,y)|\ge k$, and as a consequence, $|B\cap \mathcal{C}_G(x,y)|\ge k$. Since $B\subseteq V(G)$, we conclude that $B$ is also a $k$-adjacency generator for $G$, and consequently, $\dim_k(K_1+G)\ge\adim_k(G)$.

On the other hand, let $S$ be a $k$-adjacency basis of $G$. We will show that $S$ is also a $k$-metric basis of $K_1+G$. Since $S$ is a $k$-adjacency basis of $G$, for every $x,y\in V(G)$, we have $k\le |S\cap \mathcal{C}_G(x,y)|=|S\cap \mathcal{D}_{K_1+G}(x,y)|$. Thus, it only remains to consider the pairs $u,x$ for any $x\in V(G)$.

Given $z\in V(G)$, we define $R(z)=S\cap\left(V(G)-N_{G}(z)\right)$. Suppose for purpose of contradiction that there exists $x\in V(G)$ such that $|S\cap(V(G)-N_G(x))|\le k-1$, \textit{i.e.}, $0\le |R(x)|\le k-1$.

Now, let $F(x)=S-R(x)$. Since $|S|\ge k$, we have that $F(x)\ne\emptyset$. If $V(G)=F(x)\cup \{x\}$, then $D(G)\le 2$, which is a contradiction. If for every $y\in V(G)-\left(F(x)\cup\{x\}\right)$ there exists $z\in F(x)$ such that $d_{G}(y,z)=1$,  then $D(G)\le 4$, which is a contradiction. So, we assume that there exists a vertex $y\in V(G)-\left(F(x)\cup\{x\}\right)$ such that $d_{G}(y,z)>1$, for every $z\in F(x)$. If $V(G)=F(x)\cup\{x,y\}$, then $y\sim x$ and, as consequence, $D(G)=2$, which is also a contradiction. Hence, $V(G)-(F(x)\cup\{x,y\})\ne\emptyset$.

Since $N_{G}(y)\cap F(x)=\emptyset$ and $|R(x)|<k$, and also for any $w\in V(G)-(F(x)\cup\{x,y\})$ we have that $|\mathcal{C}_{G}(y,w)\cap S|=|\mathcal{D}_{K_1+G}(y,w)\cap S|\ge k$, we deduce that $N_{G}(w)\cap F(x)\ne \emptyset$, and this leads to $D(G)\le 5$, which is also a contradiction.

Thus, if $D(G)\ge 6$, then for every $x\in V(G)$ we have that $|R(x)|\ge k$ and, as a consequence, for every $x\in V(G)$ there exist at least $k$ vertices  $u\in S$ such that $d_{K_1+G}(u,x)=2$. Hence, $|S\cap\mathcal{D}_{K_1+G}(x,u)|\ge k$ and, a consequence, $S$ is a $k$-metric generator for $K_1+G$. Therefore, $\adim_k(G)=|S|\ge\dim_k(K_1+G)$.

Since $\adim_k(G)=\adim(\overline{G})$ for $k\in\{1,\ldots,\mathcal{C}(G)\}$, we conclude the proof.
\end{proof}

Now we make use of the lemmas above, in order to give another possible triplet satisfying Property $\mathcal{P}_1$.

\begin{lemma}\label{lemAlwaysConditions}
Let $G$ be a connected graph of order $n\ge 2$ and let $\mathcal{H}=\{H_1,\ldots,H_n\}$. If every $H_i\in\mathcal{H}$ has diameter $D(H_i)\ge 6$, then for $k\in\{1,\ldots,\mathcal{C}(G)\}$ the triplet $(G,\mathcal{H},k)$ satisfies Properties $\mathcal{P}_1$ and $\mathcal{P}_4$.
\end{lemma}

\begin{proof}
Let $v_0^i$ be the vertex of $K_1$ in $K_1+H_i$. By Lemma \ref{lemmaNoK1}, $v_0^i$ does not belong to any $k$-metric basis $A_i$ of $K_1+H_i$, \textit{i.e.}, $v_0^i\notin A_i$. Note that for every $v_r^i,v_s^i\in V(H_i)$, $\mathcal{C}_{H_i}(v_r^i,v_s^i)=\mathcal{C}_{K_1+H_i}(v_r^i,v_s^i)=\mathcal{D}_{K_1+H_i}(v_r^i,v_s^i)$. Thus, we deduce that $A_i$ is also $k$-adjacency generator for $H_i$ of cardinality $\dim_k(K_1+H_i)$. Now, from Lemma \ref{relationDim-AdimK_1}, $\adim_k(H_i)=\dim_k(K_1+H_i)$, which lead to that $A_i$ is $k$-adjacency basis of $H_i$. Note that for every $v_r^i\in V(H_i)$ it follows $\mathcal{D}_{K_1+H_i}(v_0^i,v_r^i)=(V(H_i)-N_{H_i}(v_r^i)))\cup\{v_0^i\}$. Since $v_0^i\notin A_i$, we obtain $|A_i\cap (V(H_i)-N_{H_i}(v_r^i))|\ge k$. Thus, we deduce that for any $G$ and $k\in\{1,\ldots,\mathcal{C}(G)\}$, the triplet $(G,\mathcal{H},k)$ satisfies Properties $\mathcal{P}_1$ and $\mathcal{P}_4$.
\end{proof}

Finishing this section, as we mention before, now we are able to give a result in which we describe some other classes of graphs achieving the bound of Theorem \ref{lowerBoundLexiAdj}. That is, by Corollary \ref{AllComb_TT_FT} (ii) and Lemma \ref{lemAlwaysConditions} we obtain the following.

\begin{theorem}
Let $G$ be a connected false twins free graph of order $n\ge 2$ and let $\mathcal{H}=\{H_1,\ldots,H_n\}$ be a family of graphs such that every $H_i\in\mathcal{H}$ has diameter $D(H_i)\ge 6$, then for any $k\in\{1,\ldots,\min\{\mathcal{T}(G\circ\mathcal{H}),\mathcal{C}(\mathcal{H})\}\}$, $\displaystyle\dim_k(G\circ\mathcal{H})=\sum_{i=1}^n\adim_k(H_i).$ Moreover, if $G$ is a connected true twins free graph of order $n\ge 2$, then $\displaystyle\dim_k(G\circ\overline{\mathcal{H}})=\sum_{i=1}^n\adim_k(H_i).$
\end{theorem}



\end{document}